\documentclass[12pt]{amsart}

\usepackage{amssymb,amsmath,amsthm,tikz-cd,hyperref}

\newcommand{\KK}{\mathbf{K}}
\newcommand{\mov}{\mathrm{mov}}
\newcommand{\NN}{\mathbb{N}}
\newcommand{\OO}{\mathcal{O}}
\newcommand{\PP}{\mathbb{P}}
\newcommand{\hooklongrightarrow}{\lhook\joinrel\longrightarrow}

\DeclareMathOperator{\Bs}{Bs}

\DeclareMathOperator{\vol}{vol}

\theoremstyle{plain}
\newtheorem{theorem}{Theorem}[section]
\newtheorem{lemma}[theorem]{Lemma}
\newtheorem{corollary}[theorem]{Corollary}

\theoremstyle{definition}
\newtheorem{definition}[theorem]{Definition}
\newtheorem{remark}[theorem]{Remark}
\newtheorem*{notation and conventions}{Notation and Conventions}
\newtheorem*{acknowledgment}{Acknowledgment}

\begin{document}

\mbox{}
\vspace{-1.1ex}
\title{Asymptotic constructions and invariants of graded linear series}
\author{Chih-Wei Chang}
\address{Department of Mathematics \\
National Tsing Hua University \\
Taiwan}
\email{\texttt{s101021803@m101.nthu.edu.tw}}
\author{Shin-Yao Jow}
\address{Department of Mathematics \\
National Tsing Hua University \\
Taiwan}
\email{\texttt{syjow@math.nthu.edu.tw}}
\date{}

\begin{abstract}
Let $X$ be a complete variety of dimension~$n$ over an algebraically closed field $\KK$. Let $V_\bullet$ be a graded linear series associated to a line bundle $L$ on $X$, that is, a collection $\{V_m\}_{m\in\NN}$ of vector subspaces $V_m\subseteq H^0(X,L^{\otimes m})$ such that $V_0=\KK$ and $V_k\cdot V_\ell\subseteq V_{k+\ell}$ for all $k,\ell\in\NN$. For each $m$ in the semigroup \[
 \mathbf{N}(V_\bullet)=\{m\in\NN\mid V_m\ne 0\},\]
the linear series $V_m$ defines a rational map \[
\phi_m\colon X\dashrightarrow Y_m\subseteq\PP(V_m), \]
where $Y_m$ denotes the closure of the image $\phi_m(X)$. We show that for all sufficiently large $m\in \mathbf{N}(V_\bullet)$, these rational maps $\phi_m\colon X\dashrightarrow Y_m$ are birationally equivalent, so in particular $Y_m$ are of the same dimension~$\kappa$, and if $\kappa=n$ then $\phi_m\colon X\dashrightarrow Y_m$ are generically finite of the same degree. If $\mathbf{N}(V_\bullet)\ne\{0\}$, we show that the limit \[
 \vol_\kappa(V_\bullet)=\lim_{m\in \mathbf{N}(V_\bullet)}\frac{\dim_\KK V_m}{m^\kappa/\kappa!}\]
exists, and $0<\vol_\kappa(V_\bullet)<\infty$. Moreover, if $Z\subseteq X$ is a general closed subvariety of dimension~$\kappa$, then the limit \[
 (V_\bullet^\kappa\cdot Z)_\mov=\lim_{m\in \mathbf{N}(V_\bullet)}\frac{\#\bigl((D_{m,1}\cap\cdots\cap D_{m,\kappa}\cap Z)\setminus \Bs(V_m)\bigr)}{m^\kappa}\]
exists, where $D_{m,1},\ldots,D_{m,\kappa}\in |V_m|$ are general divisors, and \[
(V_\bullet^\kappa\cdot Z)_\mov=\deg\bigl(\phi_m|_Z\colon Z\dashrightarrow \phi_m(Z)\bigr)\vol_\kappa(V_\bullet) \]
for all sufficiently large $m\in\mathbf{N}(V_\bullet)$.
\end{abstract}


\maketitle

\begin{notation and conventions}
Let $\NN$ denote the set of \emph{nonnegative} integers. We work over an algebraically closed field $\KK$ of arbitrary characteristic. A \emph{variety} is reduced and irreducible. \emph{Points} of a variety refer to closed points. If $X$ is a variety, $\KK(X)$ denotes the \emph{function field} of $X$. If $V$ is a vector space, $\PP(V)$ denotes the projective space of one-dimensional \emph{quotients} of $V$.
\end{notation and conventions}

\section{Introduction}
Let $X$ be a complete variety. Recall from \cite[Definition~2.4.1]{L} that a \emph{graded linear series} associated to a line bundle $L$ on $X$ is a collection $V_\bullet=\{V_m\}_{m\in\NN}$ of vector subspaces $V_m\subseteq H^0(X,L^{\otimes m})$ such that $V_0=\KK$ and $V_k\cdot V_\ell\subseteq V_{k+\ell}$ for all $k,\ell\in\NN$. The graded linear series $\{H^0(X,L^{\otimes m})\}_{m\in\NN}$ is called the \emph{complete graded linear series} associated to $L$. The purpose of this article is to generalize some fundamental results about asymptotic constructions for complete graded linear series to arbitrary ones. This is useful because incomplete graded linear series do naturally arise, most notably in connection with the restricted volume \cite[\S 2.4]{LM}. Our results generalize and unify several previous results in the literature: see Remark~\ref{r:special cases in lit}.

For each $m$ in the semigroup \[
 \mathbf{N}(V_\bullet)=\{m\in\NN\mid V_m\ne 0\},\]
the linear series $V_m$ defines a rational map \[
\phi_m\colon X\dashrightarrow Y_m\subseteq\PP(V_m), \]
where $Y_m$ denotes the closure of the image $\phi_m(X)$. Our first result is about the asymptotic behavior of $\phi_m$. When $V_\bullet$ is the complete graded linear series and $X$ is normal, the well-known theorem of Iitaka fibrations \cite[Theorem~2.1.33]{L} says that for all sufficiently large $m\in \mathbf{N}(V_\bullet)$, the rational maps $\phi_m\colon X\dashrightarrow Y_m$ are birationally equivalent to a fixed algebraic fiber space $\phi_\infty\colon X_\infty\longrightarrow Y_\infty$. We show that, for an arbitrary graded linear series, the same conclusion holds except that $\phi_\infty\colon X_\infty\longrightarrow Y_\infty$ may not be a fiber space.

\begin{theorem}[Asymptotic Iitaka map of graded linear series] \label{t:Iitaka for GLS}
Let $X$ be a complete variety, and let $V_\bullet$ be a graded linear series associated to a line bundle $L$ on $X$. Then there exist projective varieties $X_\infty$, $Y_\infty$, and for each $m\in \mathbf{N}(V_\bullet)$ a commutative diagram \[
\begin{tikzcd}
    X \ar[d,dashed,"\phi_m"'] & X_\infty \ar[l,"u_\infty"'] \ar[d,"\phi_\infty"] \\
    Y_m & Y_\infty \ar[l,dashed,"\nu_m"]
  \end{tikzcd}\]
of surjective morphisms and dominant rational maps, such that $u_\infty$ is birational, and $\nu_m$ is birational for all sufficiently large $m\in \mathbf{N}(V_\bullet)$.
\end{theorem}

If $V_\bullet$ is \emph{semiample}, meaning that $V_m$ is basepoint-free for some $m>0$, then the maps $\phi_{m}\colon X\longrightarrow Y_{m}$ are surjective morphisms for all $m$ in the sub-semigroup\[
 M(V_\bullet)=\{m\in\NN\mid V_m\text{ is basepoint-free}\}\subseteq \mathbf{N}(V_\bullet).\]
In this case we can conclude, as in the theorem of semiample fibrations \cite[Theorem~2.1.27]{L}, that for all sufficiently large $m\in M(V_\bullet)$, the morphisms $\phi_{m}\colon X\longrightarrow Y_{m}$ are isomorphic to a fixed morphism.

\begin{theorem}[Asymptotic Iitaka morphism of semiample graded linear series] \label{t:Iitaka for semiample GLS}
Let $X$ be a complete variety, and let $V_\bullet$ be a semiample graded linear series associated to a line bundle $L$ on $X$. Then there exist a projective variety $Y_\infty$, and for each $m\in M(V_\bullet)$ a commutative diagram
\begin{equation}\label{CD:semiample Iitaka}
\begin{tikzcd}[column sep=tiny]
& X \arrow[dl,"\phi_m"'] \arrow[dr,"\phi_\infty"] & \\
Y_m & & Y_\infty \arrow[ll,"\nu_m"]
\end{tikzcd}
\end{equation}
of surjective morphisms, such that $\nu_m$ is finite for all positive $m\in M(V_\bullet)$, and is an isomorphism for all sufficiently large $m\in M(V_\bullet)$.
\end{theorem}

An immediate corollary of Theorem~\ref{t:Iitaka for GLS} and \ref{t:Iitaka for semiample GLS} is

\begin{corollary}[Iitaka dimension and asymptotic degree]\label{c:Iitaka dim and asymp deg}
Let $X$ be a complete variety, and let $V_\bullet$ be a graded linear series associated to a line bundle $L$ on $X$ with $\mathbf{N}(V_\bullet)\ne\{0\}$.
\begin{enumerate}
 \item There is a nonnegative integer $\kappa(V_\bullet)\le \dim X$, called the \emph{Iitaka dimension} of $V_\bullet$, such that \[
 \dim\phi_m(X)\le \kappa(V_\bullet) \]
for all $m\in \mathbf{N}(V_\bullet)$, and equality holds if $m$ is sufficiently large or if $m>0$ and $V_m$ is basepoint-free. Indeed $\kappa(V_\bullet)=\dim Y_\infty$.
 \item If $\kappa(V_\bullet)=\dim X$, or equivalently $\phi_m\colon X\dashrightarrow Y_m$ is generically finite for some $m\in \mathbf{N}(V_\bullet)$, then there is a positive integer $\delta(V_\bullet)$, which we call the \emph{asymptotic degree} of $V_\bullet$, such that \[
 \deg(\phi_m\colon X\dashrightarrow Y_m)\ge \delta(V_\bullet) \]
for all $m\in \mathbf{N}(V_\bullet)$ with $\phi_m\colon X\dashrightarrow Y_m$ generically finite, and equality holds if $m$ is sufficiently large. Indeed $\delta(V_\bullet)=\deg(\phi_\infty\colon X_\infty\longrightarrow Y_\infty)$.
\end{enumerate}
\end{corollary}

The Iitaka dimension $\kappa(L)$ of a line bundle $L$ on a complete variety $X$ is defined to be the Iitaka dimension of its associated complete linear series, and $L$ is said to be \emph{big} if $\kappa(L)=\dim X$. If $L$ is big and $\dim X=n$, it is well-known that the limit \[
 \vol(L)=\lim_{m\to\infty}\frac{h^0(X,L^{\otimes m})}{m^n/n!}\]
exists \cite[Example~11.4.7]{L} and is positive \cite[Corollary~2.1.38]{L}. Moreover, this limit, called the \emph{volume} of $L$, can be approached by the ``moving intersection number'' of $n$ general divisors $E_{m,1},\ldots,E_{m,n}\in |L^{\otimes m}|$: \[
 \vol(L)=\lim_{m\to\infty}\frac{\#\bigl((E_{m,1}\cap\cdots\cap E_{m,n})\setminus \Bs|L^{\otimes m}|\bigr)}{m^n}, \]
where $\Bs|L^{\otimes m}|$ denotes the base locus of $|L^{\otimes m}|$ \cite[Theorem~11.4.11]{L}. We will generalize these results to arbitrary graded linear series $V_\bullet$. Since the usual volume of $V_\bullet$ is $0$ if $\kappa(V_\bullet)<n$, to get more interesting results we will consider the ``$\kappa(V_\bullet)$-dimensional volume'' of $V_\bullet$ and moving intersections of $V_m$ with a general $\kappa(V_\bullet)$-dimensional closed subvariety.

\begin{definition}[Moving intersection number]
Let $X$ be a complete variety, $Z\subseteq X$ be a closed subvariety, and $L$ be a line bundle on $X$. Let $V\subseteq H^0(X,L)$ be a nonzero subspace, and denote by $\Bs(V)$ the base locus of $V$. If $\dim Z=k$, the \emph{moving intersection number} of $V$ with $Z$, denoted by $(V^k\cdot Z)_\mov$, is defined by choosing $k$ general divisors $D_1,\ldots,D_k\in |V|$ and putting \[
 (V^k\cdot Z)_\mov=\#\bigl((D_1\cap\cdots\cap D_k\cap Z)\setminus \Bs(V)\bigr).\]
\end{definition}

\begin{theorem}[$\kappa$-volume and asymptotic moving intersection number] \label{t:kappa vol and asymp mov int}
Let $X$ be a complete variety, and let $V_\bullet$ be a graded linear series associated to a line bundle $L$ on $X$ with $\mathbf{N}(V_\bullet)\ne\{0\}$. Write $\kappa=\kappa(V_\bullet)$ and $n=\dim X$.
\begin{enumerate}
 \item\label{en:kappa vol}
 The limit \[
 \vol_\kappa(V_\bullet)=\lim_{m\in \mathbf{N}(V_\bullet)}\frac{\dim_\KK V_m}{m^\kappa/\kappa!}\]
exists, and $0<\vol_\kappa(V_\bullet)<\infty$. We call $\vol_\kappa(V_\bullet)$ the \emph{$\kappa$-volume} of $V_\bullet$.

 \item\label{en:asymp mov int}
 If $Z\subseteq X$ is a closed subvariety such that $\dim Z=\kappa$ and $Z\nsubseteq \Bs(V_m)$ for some $m>0$, then the limit \[
  (V_\bullet^\kappa\cdot Z)_\mov=\lim_{m\in \mathbf{N}(V_\bullet|_Z)}\frac{(V_m^\kappa\cdot Z)_\mov}{m^\kappa}\]
  exists, and $0\le (V_\bullet^\kappa\cdot Z)_\mov<\infty$. We call $(V_\bullet^\kappa\cdot Z)_\mov$ the \emph{asymptotic moving intersection number} of $V_\bullet$ with $Z$. Moreover, $(V_\bullet^\kappa\cdot Z)_\mov>0$ if and only if there exists $\ell\in \mathbf{N}(V_\bullet)$ such that $Z\nsubseteq \Bs(V_\ell)$ and $\dim\phi_\ell(Z)=\kappa$, in which case \[
      (V_\bullet^\kappa\cdot Z)_\mov=\delta(V_\bullet|_Z)\vol_\kappa(V_\bullet)=\deg\bigl(\phi_m|_Z\colon Z\dashrightarrow \phi_m(Z)\bigr)\vol_\kappa(V_\bullet) \]
  for all sufficiently large $m\in\mathbf{N}(V_\bullet)$ with $Z\nsubseteq \Bs(V_m)$. In particular, if $\kappa=n$, then \[
      (V_\bullet^n\cdot X)_\mov=\delta(V_\bullet)\vol_n(V_\bullet)=\deg(\phi_\infty\colon X_\infty\longrightarrow Y_\infty)\vol_n(V_\bullet). \]
\end{enumerate}
\end{theorem}

Note that a general $\kappa$-dimensional closed subvariety $Z\subseteq X$ satisfies $Z\nsubseteq \Bs(V_m)$ and $\dim\phi_m(Z)=\kappa$ for all sufficiently large $m\in\mathbf{N}(V_\bullet)$.

\begin{remark}[Special cases of Theorem~\ref{t:kappa vol and asymp mov int} in the literature]\label{r:special cases in lit}
Let $V_\bullet$ be a graded linear series on a complete variety $X$. We summarize some special cases of Theorem~\ref{t:kappa vol and asymp mov int} that have appeared in the literature. Note that none of them deals with graded linear series $V_\bullet$ such that $\kappa(V_\bullet)<\dim X$.
\begin{itemize}
  \item \cite[Theorem~11.4.11]{L}: as already mentioned, this is the case where $V_\bullet$ is the complete graded linear series associated to a big line bundle;
  \item \cite[Theorem~B]{ELMNP}: this deals with the case where $V_\bullet$ is the ``restricted complete linear series'' satisfying an additional assumption. The \emph{restricted complete linear series} is the restriction of a complete graded linear series from an ambient variety containing $X$. The additional assumption implies in particular that $\phi_m\colon X\dashrightarrow Y_m$ is birational for all $m\gg 0$;
  \item \cite[Theorem~C]{J}: this generalizes \cite[Theorem~B]{ELMNP} to any graded linear series such that $\phi_m\colon X\dashrightarrow Y_m$ is birational for all $m\gg 0$;
  \item \cite[Theorem~1.2~(iv)]{BP}: this generalizes \cite[Theorem~B]{ELMNP} to those restricted complete linear series such that $\phi_m\colon X\dashrightarrow Y_m$ is generically finite for all $m\gg 0$.
\end{itemize}
\end{remark}

Let us finish the introduction with a comparison between our results and some related results in \cite{KK}. Let $V_\bullet$ be a graded linear series with $\mathbf{N}(V_\bullet)\ne\{0\}$ on a complete variety $X$ of dimension~$n$. In \cite{KK}, it was shown that $\dim_\KK V_m$ grows like $m^q$ for some nonnegative integer $q$, and this growth degree $q$ was \emph{defined} to be the Iitaka dimension of $V_\bullet$ (which is different from our definition in Corollary~\ref{c:Iitaka dim and asymp deg} in terms of $\dim\phi_m(X)$). It was further shown in \cite[Corollary~3.11~(1)]{KK} that $\lim_{m\in \mathbf{N}(V_\bullet)}\dim_\KK V_m/m^q$ exists, similar to our Theorem~\ref{t:kappa vol and asymp mov int}~\eqref{en:kappa vol}. However, \cite{KK} did not show that the growth degree $q$ is equal to $\dim\phi_m(X)$ for sufficiently large $m\in \mathbf{N}(V_\bullet)$. In fact, \cite{KK} contains no discussion about asymptotic Iitaka maps, let alone results similar to our Theorem~\ref{t:Iitaka for GLS}, \ref{t:Iitaka for semiample GLS}, or Corollary~\ref{c:Iitaka dim and asymp deg}. As for our Theorem~\ref{t:kappa vol and asymp mov int}~\eqref{en:asymp mov int}, the result \cite[Theorem~4.9~(1)]{KK} might look similar, but the latter is in fact \emph{non-asymptotic} in nature: in our notation, \cite[Theorem~4.9~(1)]{KK} is saying that if $\kappa(V_\bullet)=n$ and $V_m=V_1^m$ for all $m>0$, then $(V_1^n\cdot X)_\mov=(\deg\phi_1)\vol_n(V_\bullet)$.\footnote{The notations $L$ and $\Phi_L$ in \cite{KK} correspond to our $V_1$ and $\phi_1$, respectively.} The assumption $V_m=V_1^m$ for all $m>0$ implies that $\phi_m=\phi_1$ for all $m>0$ and $(V_\bullet^n\cdot X)_\mov=(V_1^n\cdot X)_\mov$, so there is essentially no asymptotic behavior in this case. Asymptotic moving intersection numbers of arbitrary graded linear series do not appear in \cite{KK}.

\begin{acknowledgment}
The authors gratefully acknowledge the support of MoST (Ministry of Science and Technology, Taiwan).
\end{acknowledgment}

\section{Asymptotic Iitaka maps of graded linear series}
In this section, we will prove Theorem~\ref{t:Iitaka for GLS} and \ref{t:Iitaka for semiample GLS}. We start by introducing some notation.

\begin{definition}[Section ring]\label{d:section ring}
Let $X$ be a complete variety, and let $V_\bullet$ be a graded linear series associated to a line bundle $L$ on $X$. The \emph{section ring} associated to $V_\bullet$ is the graded $\KK$-algebra \[
 R(V_\bullet)=R(X,V_\bullet)=\bigoplus_{m=0}^\infty V_m.\]
\end{definition}

For each $m\in\NN$, we denote \[
 \KK[V_m]=\text{the graded $\KK$-subalgebra of $R(V_\bullet)$ generated by $V_m$.}\]
If $R$ is a graded integral domain, we write $R_{(\langle 0\rangle)}$ for the field \[
 R_{(\langle 0\rangle)}=\{a/b\mid a,b\in R\text{ are homogeneous of the same degree, }b\ne 0\}, \]
namely the degree $0$ part of the localization of $R$ that inverts all homogeneous elements not in the prime ideal $\langle 0\rangle$.

\begin{lemma}\label{l:Q((K[V_m]))=K(Y_m)}
Let $X$ be a complete variety, and let $V_\bullet$ be a graded linear series associated to a line bundle $L$ on $X$. If \[
\phi_m\colon X\dashrightarrow \PP(V_m) \]
is the rational map defined by $V_m$, and $Y_m\subseteq \PP(V_m)$ is the closure of $\phi_m(X)$, then $\KK[V_m]$ is the homogeneous coordinate ring of $Y_m$ in $\PP(V_m)$. Consequently, $\KK[V_m]_{(\langle 0\rangle)}$ is the function field of $Y_m$, that is, \[
 \KK[V_m]_{(\langle 0\rangle)}=\KK(Y_m). \]
\end{lemma}

\begin{proof}
Let $\{s_0,\ldots,s_N\}$ be a basis of $V_m$. Then there is a surjective $\KK$-algebra homomorphism from the polynomial ring $\KK[x_0,\ldots,x_N]$ to $\KK[V_m]$ sending $x_i$ to $s_i$. The kernel is generated by all homogeneous polynomials $f(x_0,\ldots,x_N)$ such that $f(s_0,\ldots,s_N)=0$ in $\KK[V_m]$. Hence \[
 \KK[V_m]\cong \KK[x_0,\ldots,x_N]/\langle f\text{ homogeneous}\mid f(s_0,\ldots,s_N)=0\text{ in }\KK[V_m]\rangle.\]

On the other hand, the rational map $\phi_m$ can be expressed as \[
\phi_m\colon X\dashrightarrow \PP(V_m)\cong \PP^N,\quad \phi_m(p)=[s_0(p):\cdots:s_N(p)] \]
for all points $p$ of $X$ in the Zariski open set \[
 U=\{p\in X\mid s_0(p),\ldots,s_N(p)\text{ are not all }0\}.\]
Since $Y_m$ is the closure of $\phi_m(U)$ in $\PP^N$, the homogeneous ideal of $Y_m$ is generated by all homogeneous polynomials $f(x_0,\ldots,x_N)$ such that \[
 f\bigl(\phi_m(p)\bigr)=f\bigl(s_0(p),\ldots,s_N(p)\bigr)=0 \]
for all $p\in U$. Since $U\subseteq X$ is Zariski open,
\begin{align*}
f\bigl(s_0(p),\ldots,s_N(p)\bigr)=0\text{ for all }p\in U &\iff f\bigl(s_0(p),\ldots,s_N(p)\bigr)=0\text{ for all }p\in X\\
  &\iff f(s_0,\ldots,s_N)=0\text{ in }\KK[V_m].
\end{align*}
Hence the homogeneous coordinate ring of $Y_m$ is precisely $\KK[V_m]$. Consequently, the function field $\KK(Y_m)=\KK[V_m]_{(\langle 0\rangle)}$ by \cite[Ch.~I, Theorem~3.4(c)]{H}.
\end{proof}

\begin{proof}[Proof of Theorem~\ref{t:Iitaka for GLS}]
We may assume that $X$ is projective by Chow's lemma. There is a natural homomorphism \[
 R(V_\bullet)_{(\langle 0\rangle)}\hooklongrightarrow \KK(X) \]
since the quotient of two sections of a line bundle is a rational function. It follows that $R(V_\bullet)_{(\langle 0\rangle)}$ is a finitely generated extension field of $\KK$ since $\KK(X)$ is. Also, for each $m\in\mathbf{N}(V_\bullet)$ there is an inclusion \[
 \KK[V_m]_{(\langle 0\rangle)}\hooklongrightarrow R(V_\bullet)_{(\langle 0\rangle)} \]
since $\KK[V_m]\subseteq R(V_\bullet)$. By \cite[Ch.~I, Theorem~4.4]{H} and Lemma~\ref{l:Q((K[V_m]))=K(Y_m)}, the homomorphisms \[
 \KK[V_m]_{(\langle 0\rangle)}\hooklongrightarrow R(V_\bullet)_{(\langle 0\rangle)}\hooklongrightarrow \KK(X) \]
induce dominant rational maps \[
 \begin{tikzcd}
   Y_m & Y_\infty \ar[l,dashed,"\nu_m"] & X \ar[l,dashed,"\varphi"]
 \end{tikzcd} \]
where $Y_\infty$ is a projective variety such that $\KK(Y_\infty)=R(V_\bullet)_{(\langle 0\rangle)}$. Moreover, there is a commutative diagram \[
\begin{tikzcd}
    X \ar[d,dashed,"\phi_m"'] \ar[dr,dashed,"\varphi"] \\
    Y_m & Y_\infty\ . \ar[l,dashed,"\nu_m"]
  \end{tikzcd}\]

Let $X_\infty\subseteq X\times Y_\infty$ be the graph of $\varphi$, that is, the closure of \[
 \{\bigl(p,\varphi(p)\bigr)\mid p\in X, \text{ $\varphi$ is regular at }p\}\subseteq X\times Y_\infty. \]
Then there is a commutative diagram \[
\begin{tikzcd}
    X \ar[dr,dashed,"\varphi"'] & X_\infty \ar[l,"u_\infty"'] \ar[d,"\phi_\infty"] \\
      & Y_\infty
  \end{tikzcd}\]
where $u_\infty$ and $\phi_\infty$ are the projections. The projection $u_\infty$ is birational since $X_\infty$ is the graph of $\varphi$. Combining the previous two commutative diagrams, we have a commutative diagram \[
\begin{tikzcd}
    X \ar[d,dashed,"\phi_m"'] & X_\infty \ar[l,"u_\infty"'] \ar[d,"\phi_\infty"] \\
    Y_m & Y_\infty \ar[l,dashed,"\nu_m"]
  \end{tikzcd}\]
of surjective morphisms and dominant rational maps for each $m\in \mathbf{N}(V_\bullet)$.

It only remains to show that if $m\in \mathbf{N}(V_\bullet)$ is sufficiently large, then $\nu_m$ is birational, or equivalently $\KK[V_m]_{(\langle 0\rangle)}=R(V_\bullet)_{(\langle 0\rangle)}$. First observe that for $i,j\in \mathbf{N}(V_\bullet)$ with $i\mid j$, it follows from $V_i^{j/i}\subseteq V_j$ that \[
  \KK[V_i]_{(\langle 0\rangle)}\subseteq \KK[V_j]_{(\langle 0\rangle)}. \]
Since $R(V_\bullet)_{(\langle 0\rangle)}=\bigcup_{m\in\mathbf{N}(V_\bullet)}\KK[V_m]_{(\langle 0\rangle)}$ is a finitely generated extension field of $\KK$, \[
  R(V_\bullet)_{(\langle 0\rangle)}=\KK[V_\ell]_{(\langle 0\rangle)} \]
for some $\ell\in\mathbf{N}(V_\bullet)$.

Let $N$ be a positive integer such that all multiples of $\gcd\mathbf{N}(V_\bullet)$ greater than or equal to $N$ appear in $\mathbf{N}(V_\bullet)$. If $m\in\mathbf{N}(V_\bullet)$ and $m\ge N+\ell$, then $m-\ell\in\mathbf{N}(V_\bullet)$, namely $V_{m-\ell}\ne 0$. Hence $\KK[V_\ell]_{(\langle 0\rangle)}\subseteq \KK[V_m]_{(\langle 0\rangle)}$ by $V_\ell\cdot V_{m-\ell}\subseteq V_m$, so \[
 \KK[V_m]_{(\langle 0\rangle)}=R(V_\bullet)_{(\langle 0\rangle)} \]
for all $m\in\mathbf{N}(V_\bullet)$ such that $m\ge N+\ell$.
\end{proof}

\begin{proof}[Proof of Theorem~\ref{t:Iitaka for semiample GLS}]
Let \[
 M(V_\bullet)_+=\{m\in M(V_\bullet)\mid m>0\}. \]
For any $i,j\in M(V_\bullet)_+$ with $i\mid j$, the image of the $(j/i)$\textsuperscript{th} symmetric power $S^{j/i}(V_{i})$ of $V_{i}$ under the multiplication map \[
 S^{j/i}(V_{i})\longrightarrow V_j \]
is a free subseries of $V_{j}$, which corresponds to a projection from $Y_{j}$ onto the $(j/i)$\textsuperscript{th} Veronese embedding of $Y_{i}$ in $\PP(V_{j})$. We denote this projection by $\pi_{i,j}$, which is a finite surjective morphism that fits into the commutative diagram
\begin{equation}\label{CD:project to Veronese}
\begin{tikzcd}[column sep=tiny]
& X \arrow[dl,"\phi_i"'] \arrow[dr,"\phi_j"] & \\
Y_i & & Y_j\ . \arrow[ll,"\pi_{i,j}"]
\end{tikzcd}
\end{equation}

Let $\phi\colon X\longrightarrow Y$ be the semiample fibration associated to the line bundle $L$ \cite[Theorem~2.1.27]{L}. Since $\phi$ is the morphism defined by the complete linear series $H^0(X,L^{\otimes m})$ for sufficiently large $m\in M(V_\bullet)$, and $\phi_m$ is defined by the subseries $V_m\subseteq H^0(X,L^{\otimes m})$, the corresponding projection $\pi_m\colon Y\longrightarrow Y_m$ is a finite surjective morphism that fits into the commutative diagram
\begin{equation}\label{CD:project from semiample fibration}
\begin{tikzcd}[column sep=tiny]
& X \arrow[dl,"\phi_m"'] \arrow[dr,"\phi"] & \\
Y_m & & Y\ . \arrow[ll,"\pi_m"]
\end{tikzcd}
\end{equation}
In fact, there is a finite surjective morphism $\pi_m$ that fits into the commutative diagram \eqref{CD:project from semiample fibration} for all $m\in M(V_\bullet)_+$, since for small $m$, one can again project onto a suitable Veronese embedding of $Y_m$. From the commutativity of \eqref{CD:project to Veronese} and \eqref{CD:project from semiample fibration} and the surjectivity of $\phi$, it follows that the diagram \[
\begin{tikzcd}[column sep=tiny]
& Y \arrow[dl,"\pi_i"'] \arrow[dr,"\pi_j"] & \\
Y_i & & Y_j \arrow[ll,"\pi_{i,j}"]
\end{tikzcd}\]
also commutes.

We claim that for each $m\in M(V_\bullet)_+$, there exists an $\ell\in M(V_\bullet)_+$ such that $m\mid \ell$, and $\pi_{\ell,k}$ are isomorphisms for all $k>0$ with $\ell\mid k$. For if this were false, there would exist an infinite sequence $m,k_1,k_2,\ldots\in M(V_\bullet)_+$ such that $m\mid k_1\mid k_2\mid\cdots$, and no $\pi_{k_i,k_{i+1}}$ is an isomorphism. Then we would have a commutative diagram \[
\begin{tikzcd}[column sep=large]
 & & Y \arrow[dll,"\pi_{k_1}"'] \arrow[dl,"\pi_{k_2}"] \arrow[d,"\pi_{k_3}"] \arrow[dr,"\cdots"] \\
Y_{k_1} & Y_{k_2} \arrow[l,"\pi_{k_1,k_2}"] & Y_{k_3} \arrow[l,"\pi_{k_2,k_3}"] & \cdots \arrow[l]
\end{tikzcd}\]
of finite surjective morphisms, with the bottom row an infinite strictly ascending chain of finite extensions, which is impossible since $\mathcal{O}_Y$ is a Noetherian $\mathcal{O}_{Y_{k_1}}$-module.

We now define the variety $Y_\infty$ and the morphisms $\nu_m\colon Y_\infty\longrightarrow Y_m$ that fulfill the statement of Theorem~\ref{t:Iitaka for semiample GLS}. Fix $m_0,\ell_0\in M(V_\bullet)_+$ such that $m_0\mid \ell_0$ and $\pi_{\ell_0,k}$ are isomorphisms for all $k>0$ with $\ell_0\mid k$. We define \[
 Y_\infty=Y_{\ell_0}\quad\text{and}\quad \phi_\infty=\phi_{\ell_0}\colon X\longrightarrow Y_\infty. \]
For each $m\in M(V_\bullet)_+$, pick an $\ell\in M(V_\bullet)_+$ such that $m\mid \ell$ and $\pi_{\ell,k}$ are isomorphisms for all $k>0$ with $\ell\mid k$, and define $\nu_m\colon Y_\infty\longrightarrow Y_m$ to be the morphism that makes the diagram \[
\begin{tikzcd}
 & Y_\infty=Y_{\ell_0} \arrow[dl,"\nu_m"'] \arrow[dr,"\pi_{\ell_0,\ell\ell_0}^{-1}"]  \\
Y_m & Y_\ell \arrow[l,"\pi_{m,\ell}"] & Y_{\ell\ell_0} \arrow[l,"\pi_{\ell,\ell\ell_0}"]
\end{tikzcd}\]
commute. The definition of $\nu_m$ does not depend on the choice of $\ell$, thanks to the commutativity of the diagram \[
\begin{tikzcd}[column sep=tiny]
& Y_k \arrow[dl,"\pi_{i,k}"'] \arrow[dr,"\pi_{j,k}"] & \\
Y_i & & Y_j \arrow[ll,"\pi_{i,j}"]
\end{tikzcd}\]
for any $i,j,k\in M(V_\bullet)_+$ with $i\mid j\mid k$, which in turn follows from the commutativity of \eqref{CD:project to Veronese}. By definition, $\nu_m$ are finite and surjective for all $m\in M(V_\bullet)_+$ since $\pi_{i,j}$ are so for all $i,j\in M(V_\bullet)_+$ with $i\mid j$. The commutativity of \eqref{CD:semiample Iitaka} follows from that of \eqref{CD:project to Veronese}.

It only remains to show that $\nu_m$ is an isomorphism if $m\in M(V_\bullet)$ is sufficiently large. Let $N$ be a positive integer such that all multiples of $\gcd M(V_\bullet)$ greater than or equal to $N$ appear in $M(V_\bullet)$. If $m\in M(V_\bullet)$ and $m\ge N+\ell_0$, then $m-\ell_0\in M(V_\bullet)$. The image of the multiplication map \[
 V_{\ell_0}\otimes V_{m-\ell_0}\longrightarrow V_m \]
is a free subseries of $V_m$, which corresponds to a projection $\mu_m$ from $Y_m$ to the Segre embedding of $Y_{\ell_0}\times Y_{m-\ell_0}$ in $\PP(V_m)$ that fits into the commutative diagram \[
\begin{tikzcd}[column sep=huge]
 & X \arrow[dl,"\phi_{\ell_0}"'] \arrow[d,"\phi_{\ell_0}\times \phi_{m-\ell_0}"] \arrow[dr,"\phi_m"]  \\
Y_{\ell_0} & Y_{\ell_0}\times Y_{m-\ell_0} \arrow[l,"pr_{Y_{\ell_0}}"] & Y_m\ . \arrow[l,"\mu_m"]
\end{tikzcd}\]
This, combined with \eqref{CD:semiample Iitaka}, shows that $\nu_m$ and $pr_{Y_{\ell_0}}\circ \mu_m$ are inverses of each other, so $\nu_m$ is an isomorphism.
\end{proof}

\section{$\kappa$-volumes and asymptotic moving intersection numbers}
In this section, we will prove Theorem~\ref{t:kappa vol and asymp mov int}. To streamline the proof, we first establish two lemmas.

\begin{lemma}\label{l:restrict series inj}
Let $X$ be a complete variety. Let $L$ be a line bundle on $X$, and let $V\subseteq H^0(X,L)$ be a nonzero subspace. Let $\phi\colon X\dashrightarrow \PP(V)$ be the rational map defined by $V$, and let $Y\subseteq \PP(V)$ be the closure of $\phi(X)$. If $Z\subseteq X$ is a subvariety not contained in the base locus $\Bs(V)$ of $V$, and $\phi(Z)$ is dense in $Y$, then the restriction $V\longrightarrow H^0(Z,L|_Z)$ is injective.
\end{lemma}

\begin{proof}
We may assume that $\phi$ is regular on $Z$ after replacing $Z$ by $Z\setminus \Bs(V)$. There is a commutative diagram \[
 \begin{tikzcd}
  Z\ar[d,"\phi|_Z"] \ar[r,hook,"\iota_Z"]& X  \ar[d,dashed,"\phi"] \\
  Y \ar[r,hook,"\iota_Y"]& \PP(V)
 \end{tikzcd}\]
where $\iota_Y$ and $\iota_Z$ are the inclusions, which induces a commutative diagram \[
 \begin{tikzcd}
  H^0(Z,L|_Z) & V \ar[l,"\iota_Z^*"] \\
  H^0\bigl(Y,\mathcal{O}_Y(1)\bigr) \ar[u,"\phi|_Z^*"] & H^0\bigl(\PP(V),\mathcal{O}(1)\bigr) \ar[l,"\iota_Y^*"] \ar[u,"\phi^*"]\ .
 \end{tikzcd}\]
Note that $\phi^*$ is an isomorphism, $\iota_Y^*$ is injective since $Y$ does not lie in a hyperplane in $\PP(V)$, and $\phi|_Z^*$ is injective since $\phi|_Z$ is dominant. Therefore $\iota_Z^*$ is injective.
\end{proof}

\begin{lemma}\label{l:dim}
Let $X$ be a complete variety. Let $M$ be a globally generated line bundle on $X$, and let $W\subseteq H^0(X,M)$ be a basepoint-free linear series. Let $Y\subseteq \PP(W)$ be the image of the morphism $X\longrightarrow \PP(W)$ defined by $W$. If $W_{\bullet}$ is the graded linear series associated to $M$ on $X$ given by \[
  W_{k}=\operatorname{Im}\bigl(S^k(W)\longrightarrow H^0(X,M^{\otimes k})\bigr), \]
that is, the image of the $k$\textsuperscript{th} symmetric power of $W$ under the multiplication map, then \[
 \dim_\KK W_k=h^0\bigl(Y,\OO_Y(k)\bigr) \]
for all sufficiently large $k$.
\end{lemma}

\begin{proof}
Write the morphism $X\longrightarrow \PP(W)$ defined by $W$ as a composition \[
 \begin{tikzcd}
  X\ar[r,two heads,"\varphi"] & Y \ar[r,hook,"\iota"] & \PP(W)\ ,
 \end{tikzcd}\]
where $\iota$ denotes the inclusion. Then \[
 M=(\iota\circ\varphi)^* \OO_{\PP(W)}(1)=\varphi^* \OO_{Y}(1),\]
and \[
 W=(\iota\circ\varphi)^*H^0\bigl(\PP(W),\OO_{\PP(W)}(1)\bigr).\]
Hence
\begin{align*}
 W_{k}&=(\iota\circ\varphi)^*\operatorname{Im}\Bigl(S^kH^0\bigl(\PP(W),\OO_{\PP(W)}(1)\bigr)\longrightarrow H^0\bigl(\PP(W),\OO_{\PP(W)}(k)\bigr)\Bigr) \\
  &=(\iota\circ\varphi)^* H^0\bigl(\PP(W),\OO_{\PP(W)}(k)\bigr)\\
  &=\varphi^*\Bigl(\iota^*H^0\bigl(\PP(W),\OO_{\PP(W)}(k)\bigr)\Bigr).
\end{align*}

Let $\mathcal{I}_Y$ denote the ideal sheaf of $Y$ on $\PP(W)$. The short exact sequence \[
 0\longrightarrow \mathcal{I}_Y(k)\longrightarrow \OO_{\PP(W)}(k)\longrightarrow \OO_{Y}(k)\longrightarrow 0 \]
of sheaves on $\PP(W)$ induces a long exact sequence \[
 \cdots\longrightarrow H^0\bigl(\PP(W),\OO_{\PP(W)}(k)\bigr)\xrightarrow{\ \iota^*\ } H^0\bigl(Y,\OO_{Y}(k)\bigr)\longrightarrow H^1\bigl(\PP(W),\mathcal{I}_Y(k)\bigr)\longrightarrow\cdots \]
of cohomology groups. If $k$ is sufficiently large, then $H^1\bigl(\PP(W),\mathcal{I}_Y(k)\bigr)=0$ by Serre's vanishing theorem, so \[
 \iota^*H^0\bigl(\PP(W),\OO_{\PP(W)}(k)\bigr)=H^0\bigl(Y,\OO_{Y}(k)\bigr).\]
Hence \[
 W_k=\varphi^*H^0\bigl(Y,\OO_{Y}(k)\bigr).\]
Since $\varphi\colon X\longrightarrow Y$ is onto, $\varphi^*$ is one-to-one. Thus \[
 \dim_\KK W_k=\dim_\KK H^0\bigl(Y,\OO_Y(k)\bigr)=h^0\bigl(Y,\OO_Y(k)\bigr). \]
\end{proof}

\begin{proof}[Proof of Theorem~\ref{t:kappa vol and asymp mov int}]
For each $m\in \mathbf{N}(V_\bullet)$, let \[
 \pi_m \colon X_m\longrightarrow X \]
be the blowing-up of $X$ with respect to the base ideal $\mathfrak{b}_m$ of $V_m$. Then \[
 \mathfrak{b}_m\OO_{X_m}=\OO_{X_m}(-F_m)\]
for some effective Cartier divisor $F_m$ on $X_m$. Let $M_m$ denote the globally generated line bundle \[
  M_m=\pi^*_m L^{\otimes m} \otimes \OO_{X_m}(-F_m) \]
on $X_m$. Then \[
  \pi^*_m |V_m|= |W_m| + F_m, \]
where $W_m\subseteq H^0(X_m,M_m)$ is a basepoint-free linear series that defines a morphism $\varphi_m\colon X_m\longrightarrow Y_m$ fitting into the commutative diagram \[
 \begin{tikzcd}
   X_m \ar[d,"\pi_m"'] \ar[dr,"\varphi_m",end anchor=north west] \\
   X \ar[r,dashed,"\phi_m"']& Y_m\subseteq \PP_m\ ,
 \end{tikzcd}\]
where $\PP_m=\PP(V_m)=\PP(W_m)$.

If $Z\nsubseteq \Bs(V_m)$ and $Z_m\subseteq X_m$ is the strict transform of $Z$, then
\begin{equation}\label{eq:mov int as usual int}
 (V_m^\kappa\cdot Z)_\mov=M_m^\kappa\cdot Z_m=\bigl(\varphi_m^* \OO_{Y_m}(1)\bigr)^\kappa\cdot Z_m=\OO_{Y_m}(1)^\kappa\cdot ({\varphi_m}_*Z_m).
\end{equation}
Since ${\varphi_m}_*Z_m=0$ if $\dim\varphi_m(Z_m)=\dim\phi_m(Z)<\kappa$, \[
 (V_m^\kappa\cdot Z)_\mov=0 \]
for all $m\in \mathbf{N}(V_\bullet|_Z)$ if $\kappa(V_\bullet|_Z)<\kappa$. Thus we may assume that $\kappa(V_\bullet|_Z)=\kappa$.

Since $\kappa(V_\bullet|_Z)=\kappa$, by Lemma~\ref{l:restrict series inj}, \[
 \dim_\KK V_m=\dim_\KK V_m|_Z \]
for all sufficiently large $m\in\mathbf{N}(V_\bullet|_Z)$. By definition, we also have \[
 (V_m^\kappa\cdot Z)_\mov=\bigl((V_m|_Z)^\kappa\cdot Z\bigr)_\mov \]
for all $m\in\mathbf{N}(V_\bullet|_Z)$. Hence it is enough to prove Theorem~\ref{t:kappa vol and asymp mov int} for the graded linear series $V_\bullet|_Z$ on $Z$. Since $\kappa(V_\bullet|_Z)=\dim Z$, this means that we may reduce Theorem~\ref{t:kappa vol and asymp mov int} to the case $\kappa=n$.

For each $m\in \mathbf{N}(V_\bullet)$, consider the graded linear series $W_{m,\bullet}$ associated to the line bundle $M_m$ on $X_m$ given by \[
  W_{m,k}=\operatorname{Im}\bigl(S^k(W_m)\longrightarrow H^0(X_m,M_m^{\otimes k})\bigr). \]
By Lemma~\ref{l:dim}, \[
 \dim_\KK W_{m,k}=h^0\bigl(Y_m,\OO_{Y_m}(k)\bigr) \]
for all sufficiently large $k$. Hence if $m$ is large enough so that $\dim Y_m=\kappa=n$, then \[
 \vol_n(W_{m,\bullet})=\lim_{k\to\infty}\frac{\dim_\KK W_{m,k}}{k^n/n!}=\lim_{k\to\infty}\frac{h^0\bigl(Y_m,\OO_{Y_m}(k)\bigr)}{k^n/n!}=\vol\bigl(\OO_{Y_m}(1)\bigr)=\OO_{Y_m}(1)^n. \]
Using \eqref{eq:mov int as usual int} in the case $\kappa=n$, $Z=X$, and $Z_m=X_m$, we get
\begin{equation}\label{eq:main}
  \vol_n(W_{m,\bullet})=\OO_{Y_m}(1)^n=\frac{(V_m^n\cdot X)_\mov}{\deg(\phi_m\colon X\dashrightarrow Y_m)}=\frac{(V_m^n\cdot X)_\mov}{\delta(V_\bullet)}
\end{equation}
for all sufficiently large $m\in \mathbf{N}(V_\bullet)$.

It is proved in \cite{KK} that the usual volume \[
 \vol(V_\bullet)=\vol_n(V_\bullet)=\lim_{m\in \mathbf{N}(V_\bullet)}\frac{\dim_\KK V_m}{m^n/n!} \]
of any graded linear series $V_\bullet$ exists (and is finite), and the following version of Fujita's approximation theorem holds: for any $\epsilon>0$, there exists an integer $m_0=m_0(\epsilon)$, such that if $m\in \mathbf{N}(V_\bullet)$ and $m\ge m_0$, then  \[
   \vol(V_\bullet)-\epsilon\le\frac{\vol(W_{m,\bullet})}{m^n}\le\vol(V_\bullet). \]
This implies that $\vol(V_\bullet)>0$ since $\vol(W_{m,\bullet})=\OO_{Y_m}(1)^n>0$ for all sufficiently large $m\in \mathbf{N}(V_\bullet)$ by \eqref{eq:main}. It also implies that \[
 \lim_{m\in \mathbf{N}(V_\bullet)}\frac{\vol(W_{m,\bullet})}{m^n}= \vol(V_\bullet). \]
So by \eqref{eq:main} again, \[
(V_\bullet^n\cdot X)_\mov=\lim_{m\in \mathbf{N}(V_\bullet)}\frac{(V_m^n\cdot X)_\mov}{m^n}=\lim_{m\in \mathbf{N}(V_\bullet)}\frac{\delta(V_\bullet)\vol(W_{m,\bullet})}{m^n}=\delta(V_\bullet)\vol(V_\bullet).\]
\end{proof}

\end{document}